\documentclass{amsart}
\usepackage{amsfonts,amssymb,amsmath,amsthm}
\usepackage{url}
\usepackage{enumerate}

\makeatletter
\@namedef{subjclassname@2010}{  \textup{2010} Mathematics Subject Classification.}
\makeatother
\newtheorem{thrm}{Theorem}[section]

\theoremstyle{definition}

\newtheorem{remark}[thrm]{Remark}
\numberwithin{equation}{section}
\email{bmelmostafa@hotmail.com}
\email{hzoubeir2014@gmail.com}
\input{tcilatex}

\begin{document}
\address{ }
\author{Elmostafa Bendib}
\address{Ibn Tofail University, Department of Mathematics\\
Faculty of Sciences, P. O. B : $133,$ Kenitra, Morocco.}
\author{Hicham Zoubeir}
\address{Ibn Tofail University, Department of Mathematics\\
Faculty of Sciences, P. O. B : $133,$ Kenitra, Morocco.}
\title[Solvability in Gevrey classes of some linear functional equations]{%
Solvability in Gevrey classes of some linear functional equations}

\begin{abstract}
In this paper, we associate to each positive number $k$ a new class of
endomorphisms of the sheaf of germs of holomorphic functions on $[-1,1]$ and
prove the solvability in the Gevrey class $G_{k}\left( \left[ -1,1\right]
\right) $ of some linear functional equations related to these linear
endomorphisms.
\end{abstract}

\dedicatory{$\emph{This}$ $\emph{modest}$ $\emph{work}$ $\emph{is}$ $\emph{%
dedicated}$ $\emph{to}$ $\emph{the}$ $\emph{memory}$ $\emph{of}$ $\emph{our}$
$\emph{beloved}$ $\emph{master}$ $\emph{Ahmed}$ $\emph{Intissar}$ $\emph{%
(1951-2017),}$ $\emph{a}$ $\emph{distinguished}$ $\emph{professor,a}$ $\emph{%
brilliant}$ $\emph{mathematician,a}$ $\emph{man}$ $\emph{with}$ $\emph{a}$ $%
\emph{golden}$ $\emph{heart.}$\emph{\ }}
\maketitle

\section{Introduction}

The functional equations have been the subject\ of intensive studies because
of their relation to applied and social sciences. The extreme variety of
areas where functional equations are found only enhance their
attractiveness. In the study of such equations there are different
approaches and various research directions (cf for example (\cite{ACZ1})-(%
\cite{BELI}), (\cite{CHE})-(\cite{COR}), (\cite{KUC})). However in our
opinion there are a few studies on their solvability in Gevrey classes. In
this paper, we associate to each number $k>0$ a new class of endomorphisms
of the sheaf of germs of holomorphic functions on $[-1,1]$ and prove the
solvability in a Gevrey class of linear functional equations related to
these endomorphisms. We apply then the result obtained to prove the
solvability in the Gevrey class $G_{k}([-1,1])$ of some linear functional
equations.

\section{Notations, definitions and preleminaries}

Let $S$ be a nonempty subsets of $\mathbb{C}$ and $f\colon S\longrightarrow 
\mathbb{C}$ a bounded function. $\Vert f\Vert _{\infty ,S}$ denotes the
quantity :%
\begin{equation*}
\Vert f\Vert _{\infty ,S}=\sup_{z\in S}|f(z)|
\end{equation*}

For $z\in 
\mathbb{C}
$ we set $\varrho (z,S$ $):=\underset{u\in S}{\inf }|z-u|.$

$O(S)$ denotes the set of holomorphic functions on some neighborhood of $S.$

For $z\in 
\mathbb{C}
$ and $h>0$, $B(z,h)$ is the open ball in $%
\mathbb{C}
$ with center $z$ and radius $h$ .

For $r>0$, $k>0$, $A>0$ and $n\in 
\mathbb{N}
^{\ast }$, we set :%
\begin{equation*}
\left \{ 
\begin{array}{c}
\lbrack -1,1]_{r}=[-1,1]+B\text{ }(0,r) \\ 
\lbrack -1,1]_{k,A,n}:=[-1,1]+B(0;An^{-\frac{1}{k}})%
\end{array}%
\right.
\end{equation*}

Thus we have :%
\begin{equation*}
\left \{ 
\begin{array}{c}
\lbrack -1,1]_{r}=\{z\in 
\mathbb{C}
:\varrho (z,[-1,1])<r\} \\ 
\lbrack -1,1]_{k,A,n}=\{z\in C:\varrho (z;[-1,1])<An^{\frac{-1}{k}}\}%
\end{array}%
\right.
\end{equation*}

Let $E$ \ be a non empty set and $g\colon E\longrightarrow E$ a mapping.For
every $n\in \mathbb{N},g^{<n>}$denotes the iterate of order $n$ of $g$ for
the composition of mappings.

Through this paper the real number $k>0$ will be a fixed\ constant real
number.

The Gevrey class $G_{k}$ $([-1,1])$ is the set of all functions $f\colon
\lbrack -1,1]\longrightarrow \mathbb{C}$ of class $C^{\infty }$ on $[-1,1]$
such that there exist some constants $C>0,A>0$ verifying the following
inequalities :%
\begin{equation*}
\Vert f^{\text{ }(n)}\Vert _{\infty ,[-1,1]}\leq CA^{n}n^{n(1+\frac{1}{k})},%
\text{ }n\in 
\mathbb{N}%
\end{equation*}%
with the convention that $0^{0}=1.$

Let $\psi $ be a holomorphic function on a neighborhood of $[-1,1]$ such
that $\psi ([-1,1]\subset \lbrack -1,1]$. We set :%
\begin{equation*}
\lambda (\psi ):=\sup_{n\geq 1}\left( \frac{\parallel \psi ^{(n)}\parallel
_{\infty ,[-1,1]}}{n!}\right) ^{\frac{1}{n}}
\end{equation*}%
Let us observe that :%
\begin{equation*}
\lambda (\psi )<+\infty
\end{equation*}

A sequence $(f_{n})_{n\geq 1}$ of germes of holomorphic functions on $%
[-1,1]. $ $(f_{n})_{n\geq 1}$ is called $k-$sequence if there exists $B>0$
such that the following conditions hold for every $n\geq 1:$%
\begin{equation*}
\left \{ 
\begin{array}{c}
f_{n}\in O([-1,1]_{k,B,n\text{ }}) \\ 
||f_{n}||_{\infty ,[-1,1]_{k,A,n\text{ }}}\leq C\text{ }\theta ^{n}\text{ }%
\end{array}%
\right.
\end{equation*}%
where $C>0$ and $\theta \in ]0,1[$ are constants.

The following result which is a direct consequence of a theorem stated in (%
\cite{BELG}, page 223), point out the link betwen the set of $k-$sequences
and the Gevrey class $G_{k}([-1,1])$.

\begin{theorem}
\textit{A function }$F$\textit{\ of class }$C^{\infty }$\textit{on }$[-1;1]$%
\textit{\ belongs to }$G_{k}([-1;1])$\textit{\ if and only if\ there exists
a }$k-$\textit{sequence }$(g_{n})_{n\geq 1}$\textit{\  \ we have that }$:$%
\begin{equation*}
F=\underset{n=1}{\overset{+\infty }{\sum }}g_{n}|_{[-1;1]}
\end{equation*}
\end{theorem}

Let $\tciFourier :=(\varphi _{p})_{p\in 
\mathbb{N}
}$ be a sequence of functions $\varphi _{p}:[-1,1]_{\sigma }\rightarrow 
\mathbb{C}
$ $\left( \sigma >0\right) $. We say that $\tciFourier $ verifies the $E(k)$
property if there exists a constant $\tau \in $ $]0,\sigma ]$ depending only
on $\tciFourier $ such that for all $A$ in $]0,\tau ]$ there exists an
integer $M(A)$ depending only on $A$ such that we have :%
\begin{equation}
\varphi _{p}([-1,1]_{k,A,n+1})\subset \lbrack -1,1]_{k,A,n\text{ }},n\geq
M(A),\text{ }p\in 
\mathbb{N}
\label{(2.2)}
\end{equation}%
The real $\tau $ is then called a $k$-threshold for the family $\tciFourier $%
.

\begin{remark}
If the family $\tciFourier $ verify the $\mathfrak{A}(k)$ property then%
\begin{equation*}
\varphi _{p}([-1,1])\subset \lbrack -1,1],\text{ }p\in 
\mathbb{N}
\text{ \  \ }
\end{equation*}
\end{remark}

An endomorphism $T$ of the sheaf $O([-1,1])$ of holomorphic functions on $%
[-1,1]$ is said to verify the $\mathfrak{A}(k)$ property if there exist two
constants $\tau >0,\rho \in ]0,1[$ depending only on $T$ such that for all $%
A $ in $]0,\tau ]$ there exists an integer $N(A)$ depending only on $A$
satisfying the following properties :%
\begin{equation}
T(O([-1,1]_{k,n,A})\subset O([-1,1]_{k,n+1,A},\text{ }n\geq N(A)
\label{(2.3)}
\end{equation}%
\begin{equation}
||T(f)||_{\infty ,[-1,1]_{k,n+1,A}}\leq \rho ||f||_{\infty ,[-1,1]_{k,n,A}}\
\ ,n\geq N(A),\text{ }f\in O([-1,1]_{k,n,A})  \label{(2.4)}
\end{equation}%
\begin{equation}
||T(f\text{ })|_{[-1,1]}||_{\infty ,[-1,1]}\leq \rho ||f\text{ }%
|_{[-1,1]}||_{\infty ,[-1,1]},\text{ }f\in O([-1,1])  \label{(2.5)}
\end{equation}%
The real $\tau $ is then called a $k$-threshold for the endomorphism $T$.

An endomorphism $T$ of the sheaf $O([-1,1])$ which verifies the $\mathfrak{A}%
(k)$ property has the following fundamental property.

\begin{proposition}
\textit{Let }$T$\textit{\  \ be an endomorphism of the sheaf }$O([-1,1])$%
\textit{\ which verify the }$\mathfrak{A}(k)$\textit{\ property. Then }$T$%
\textit{\ induces a unique endomorphism }$\widetilde{T\text{ }}$\textit{of
the Gevrey class }$G_{k}([-1,1])$\textit{\ such that for evey }$k-$\textit{%
sequence }$(f_{n})_{n\geq 1}$\textit{\ the sequence }$(T(f_{n}))_{n\geq 1}$%
\textit{\ is also a }$k-$\textit{sequence and the following condition holds }%
$:$%
\begin{equation}
\widetilde{T\text{ }}(\underset{n=1}{\overset{+\infty }{\sum }}%
f_{n}|_{[-1,1]})=\underset{n=1}{\overset{+\infty }{\sum }}T(f_{n})|_{[-1,1]}
\label{restriction}
\end{equation}
\end{proposition}

\begin{proof}
Let $f$ $\in G_{k}([-1,1]).$ Let $(g_{n})_{n\geq 1}$ and $(h_{n})_{n\geq 1}$
be $k-$sequences such that :%
\begin{equation*}
f=\underset{n=1}{\overset{+\infty }{\sum }}g_{n}|_{[-1,1]}=\underset{n=1}{%
\overset{+\infty }{\sum }}h_{n}|_{[-1,1]}
\end{equation*}%
Then there exist $A_{1}>0,$ $C_{1}>0,$ $0<\theta _{1}<1$ such that the
following conditions hold for all $n\geq 1:$%
\begin{equation*}
\left \{ 
\begin{array}{c}
g_{n},h_{n}\in O([-1,1]_{k,A_{1},n\text{ }}) \\ 
\max \left( ||g_{n}||_{\infty ,[-1,1]_{k,A_{1},n\text{ }}},||h_{n}||_{\infty
,[-1,1]_{k,A_{1},n\text{ }}}\right) \leq C_{1}\theta _{1}^{n}%
\end{array}%
\right.
\end{equation*}%
Let us set for all $n\geq 1:$%
\begin{equation*}
\left \{ 
\begin{array}{c}
w_{n}:=g_{n}|_{[-1,1]_{k,A_{1},n\text{ }}-}h_{n}|_{[-1,1]_{k,A_{1},n\text{ }%
}} \\ 
\Phi _{n}:=\overset{n}{\underset{j=1}{\sum }}w_{j}%
\end{array}%
\right.
\end{equation*}%
We have for all $n\geq 1:$%
\begin{eqnarray*}
&&T(\Phi _{n})|_{[-1,1]} \\
&=&\overset{n}{\underset{j=1}{\sum }}T(w_{j})|_{[-1,1]} \\
&=&\overset{n}{\underset{j=1}{\sum }}T(g_{j})|_{[-1,1]}-\overset{n}{\underset%
{j=1}{\sum }}T(h_{j})|_{[-1,1]}
\end{eqnarray*}%
Since the sequence of functions $(\Phi _{n})_{n\geq 1}$is uniformly
convergent on $[-1,1]$ to the null function, it follows from the condition (%
\ref{(2.5)}) that the function series $\sum T(g_{n})|_{[-1,1]}$ and $\sum
T(h_{n})|_{[-1,1]}$ are uniformly convergent\ on $[-1,1]$ to the same
function which we denote by $\widetilde{T\text{ }}(f$ $).$ The mapping $%
\widetilde{T\text{ \ }}$ is well defined and linear on the Gevrey class $%
G_{k}([-1,1]).$ Furthermore $\widetilde{T\text{ \ }}$ is the unique
endomorphism on $G_{k}([-1,1])$ which satisfies the condition (\ref%
{restriction})$.$
\end{proof}

\section{Statement of the main result and of its corollary}

\begin{theorem}
\textit{Let }$T$\textit{\ an endomorphism on the sheaf }$O([-1,1])$\textit{\
which verify the }$\mathfrak{A}(k)$\textit{\ property. Then for every }$u\in
O([-1,1])$\textit{\ the function series }$\sum T^{\text{ }<n>}(u)|_{[-1,1]}$%
\textit{\ is uniformly convergent on }$[-1,1]$\textit{\ and its sum }$v$%
\textit{\ is a solution of the linear functional equation }%
\begin{equation}
\Phi -\widetilde{T\text{ }}(\Phi )=u  \label{(3.1)}
\end{equation}%
\textit{which belongs to the Gevrey class }$G_{k}([-1,1]).$
\end{theorem}

\begin{corollary}
\textit{Let }$a:=$\textit{\ }$(a_{n})_{n\geq 0}$\textit{\ be a sequence of
holomorphic functions on }$[-1,1]_{\sigma }$ $(\sigma >0)$\textit{\ and }$%
\varphi :=(\varphi _{n})_{n\geq 0}$\textit{\ a sequence which verify the }$%
E(k)$\textit{\ property. Assume that }$:$%
\begin{equation}
\overset{+\infty }{\underset{n=0}{\sum }}||a_{n}||_{\infty ,[-1,1]_{\sigma
}}<1  \label{(3.2)}
\end{equation}%
\textit{Then the linear functional equation }$:$%
\begin{equation}
\Phi (x)-\overset{+\infty }{\underset{n=0}{\sum }}a_{n}(x)\Phi (\varphi
_{n}(x))=u(x)  \label{(3.3)}
\end{equation}%
\textit{has a unique solution which furthermore belongs to the Gevrey class }%
$G_{k}([-1,1]).$
\end{corollary}

\section{Proof$\ $of$\ $the$\ $main\ result and of its corollary}

\subsection{Proof of the main result}

Let $\tau $ be a $k$-threshold for the endomorphism $T$. Let $r\in ]0,\tau ]$
such that $u\in O([-1,1]_{r}).$ Consider the sequence of functions $%
(T^{\left \langle n\right \rangle }(u))_{n\geq 0}.$Then there exists, thanks
to the conditions (\ref{(2.2)}) and (\ref{(2.3)}) an integer $N$ such that :%
\begin{equation*}
\left \{ 
\begin{array}{c}
T(O([-1,1]_{k,n,r})\subset O([-1,1]_{k,n+1,r},\text{ }n\geq N \\ 
||T(f\text{ })||_{\infty ,[-1,1],r}\leq \rho ||f\text{ }||_{\infty
,[-1,1]_{k,n,r}},\text{ }n\geq N,\text{ }f\in O([-1,1]_{k,n,r}%
\end{array}%
\right.
\end{equation*}

It is then clear that we have for all $n\geq N:$%
\begin{equation*}
\left \{ 
\begin{array}{c}
T^{\left \langle n\right \rangle }(u)\in O([-1,1]_{k,n,r}) \\ 
||T^{\left \langle n\right \rangle }(u)||_{\infty ,[-1,1]_{k,,n,r}}\leq
||T(u)||_{\infty ,[-1,1]_{k,N,r}}.\rho ^{n-N}%
\end{array}%
\right.
\end{equation*}

It follows, by virtue of theorem 1, that the function series $\sum
T^{\left
\langle n\right \rangle }(u)|_{[-1,1]}$ is uniformly convergent to
a function $w\in G_{k}([-1,1])$ which is a solution of the equation (\ref%
{(3.1)}).

$\square $

\subsection{Proof of the corollary}

Let $f\in O([-1,1])$,then there exists $\alpha \in ]0,\sigma ]$ such that $%
f\in O([-1,1]_{\alpha })$ where $\sigma $\ is a $k-$threshold of the
sequence $\varphi .$ Then there exists $M\in 
\mathbb{N}
^{\ast }$such that :%
\begin{equation*}
\varphi _{p}([-1,1]_{k,\alpha ,n+1})\subset \lbrack -1,1]_{k,\alpha ,n\text{ 
}},\text{ }n\geq M,\text{ }p\in 
\mathbb{N}%
\end{equation*}%
It follows that :%
\begin{equation*}
\left \{ 
\begin{array}{c}
f\circ \varphi _{p}\in O([-1,1]_{k,M\text{ }+1\text{ ,}\alpha }),\text{ }%
p\in 
\mathbb{N}
\\ 
|a_{p}(z)f\text{ }[\varphi _{p}(z)]|\leq ||a_{p}||_{\infty ,[-1,1]_{\sigma
}}||f\text{ }||_{\infty ,[-1,1]_{k,M\text{ }+1,\alpha \text{ }}},\text{ }%
z\in \lbrack -1,1]_{k,M\text{ }+1,\alpha \text{ }}%
\end{array}%
\right.
\end{equation*}%
Thence from the condition (\ref{(3.2)}) entails that the function series $%
\sum a_{n}.(f\circ \varphi _{n})$ is uniformly convergent on $%
[-1,1]_{k,\alpha ,M\text{ }+1\text{ }}.$ Thence if we set : 
\begin{equation*}
T_{1}(f\text{ }):=\underset{n=1}{\overset{+\infty }{\sum }}a_{n}.(f\circ
\varphi _{n})
\end{equation*}%
we define an endomorphism $T_{1}$ of $O([-1,1]).$ Let us prove that $T_{1}$
verify the $\mathfrak{A}(k)$ property. Let $A\in ]0,\sigma ]$ , then there
exists an integer $M(A)$ depending only on $A$ satisfying the following
properties :%
\begin{equation*}
\varphi _{p}([-1,1]_{k,m\text{ }+1,A})\subset O([-1,1]_{k,m,A}\text{ },\text{
}m\geq M(A)\text{ }
\end{equation*}%
Let $v\in O([-1,1]_{k,n,A})$ where $n$ is an integer such that $n\geq M(A).$
Then : 
\begin{equation*}
\left \{ 
\begin{array}{c}
v\circ \varphi _{p}\in O([-1,1]_{k,n+1,A}),\text{ }p\in 
\mathbb{N}
\\ 
||a_{p}(u\text{ }\circ \varphi _{p})||_{\infty ,[-1,1]_{k,n+1,A}}\leq
||a_{p}||_{\infty ,[-1,1]_{\sigma }}||u||_{\infty ,[-1,1]_{k,,n,A\text{ }}},%
\text{ }p\in 
\mathbb{N}%
\end{array}%
\right.
\end{equation*}%
It follows that the following facts hold for each integer $n\geq M(A):$ 
\begin{equation*}
\left \{ 
\begin{array}{c}
T_{1}(v)\in O([-1,1]_{k,n+1,A}) \\ 
||T_{1}(v)||_{\infty ,[-1,1]_{k,n+1,A}}\leq \left( \overset{+\infty }{%
\underset{p=0}{\sum }}||a_{p}||_{\infty ,[-1,1]_{\sigma }}\right)
||u||_{\infty ,[-1,1]_{k,,n,A\text{ }}}%
\end{array}%
\right.
\end{equation*}%
On the other hand we have for every $x\in \lbrack -1,1]:$

\begin{equation*}
|T_{1}(v)(x)|\leq \left( \overset{+\infty }{\underset{p=0}{\sum }}%
||a_{p}||_{\infty ,[-1,1]_{\sigma }}\right) |v\text{ }(\varphi _{p}(x))|
\end{equation*}%
It follows, from the remark 1, that :%
\begin{equation*}
||T_{1}(v)|_{[-1,1]}||_{\infty ,[-1,1]}\leq \left( \overset{+\infty }{%
\underset{p=0}{\sum }}||a_{p}||_{\infty ,[-1,1]_{\sigma }}\right) ||v\text{ }%
||_{\infty ,[-1,1]}
\end{equation*}%
Concequently since $\overset{+\infty }{\underset{p=0}{\sum }}%
||a_{p}||_{\infty ,[-1,1]_{\sigma }}<1,$we conclude that the endomorphism $%
T_{1}$ satisfy the $\mathfrak{A}(k)$ property. Thence thanks to the main
result the linear functional equation (\ref{(3.3)}) has for every $u\in
O([-1,1])$\ a unique solution which furthermore belongs to the Gevrey class $%
G_{k}([-1,1]).$

$\square $

\section{Some examples}

We need first to prove some useful propositions.

\begin{proposition}
\textit{Let }$\psi $\textit{\ be a holomorphic function on a neighborhood of 
}$[-1,1]$\textit{\ such that }$\psi ([-1,1]\subset \lbrack -1,1]$\textit{.
Assume that }$\lambda (\psi )\leq 1.$ \textit{Then the function }$\psi $%
\textit{\ verify the }$E(1)$\textit{\ property.}
\end{proposition}

\begin{proof}
Let $\sigma >0$ be such that $\psi \in O([-1,1]_{\sigma }).$ Let $A\in
]0;\min (\frac{1}{2},\sigma )[$, $p\in 
\mathbb{N}
^{\ast }$and $z\in \lbrack -1,1]_{k,p+1,A}.$Let $\widehat{z}$ the closest
point of $[-1,1]$ to $z.$ We have the following inequalities :%
\begin{eqnarray*}
\varrho (\psi (z),[-1,1]) &\leq &|\psi (z)-\psi (\widehat{z})| \\
&\leq &\overset{+\infty }{\underset{j=1}{\sum }}\frac{|\psi ^{(j)}(\widehat{z%
}))|}{j!}|z-\widehat{z}|^{j} \\
&\leq &\overset{+\infty }{\underset{j=1}{\sum }}|z-\widehat{z})|^{j} \\
&\leq &\overset{+\infty }{\underset{j=1}{\sum }}\varrho (z,[-1,1])^{j} \\
&\leq &\frac{A}{p+1-A} \\
&<&\frac{A}{p}
\end{eqnarray*}%
It follows that : 
\begin{equation*}
\psi ([-1,1]_{1,p+1,A})\subset \lbrack -1,1]_{1,p,A}\text{ },p\in 
\mathbb{N}
^{\ast }
\end{equation*}%
Thence the function $\psi $ has the $E(1)$ property.
\end{proof}

\begin{proposition}
\textit{Let }$g$\textit{\  \ be an entire function such that}%
\begin{equation*}
\left \{ 
\begin{array}{c}
g([-1,1])\subset \lbrack -1,1] \\ 
\lambda (g)\leq 1%
\end{array}%
\right. 
\end{equation*}%
\textit{Let }$(P_{n})_{n\geq 1}$\textit{\ a sequence of holomorphic
functions on }$[-1;1]_{\sigma }(\sigma >0)$\textit{\ such that we have for
every }$n\in 
\mathbb{N}
^{\ast }$%
\begin{eqnarray*}
P_{n}([-1,1]) &\subset &[-1,1] \\
\lambda (P_{n}) &\leq &1
\end{eqnarray*}%
$(g_{n})_{n\geq 1}$ \textit{is the sequence of functions }$g_{n}:%
\mathbb{C}
\rightarrow 
\mathbb{C}
$\textit{\  \ defined by the formula}%
\begin{equation*}
g_{n}(z):=g^{\left \langle n\right \rangle }\left( \frac{z}{2^{n-1}}\right) 
\end{equation*}%
\textit{The sequences of functions }$(g\circ P_{n})_{n\geq 1}$\textit{\ and }%
$(g_{n})_{n\geq 1}$ \textit{have the }$E(1)$\textit{\ property.}
\end{proposition}

\begin{proof}
1- For every $n\in 
\mathbb{N}
^{\ast }$ the function $g\circ P_{n}\in O([-1,1]_{\sigma }).$\  \ Let $A\in
]0,\min (\frac{1}{2},\sigma )[$, $p\in 
\mathbb{N}
^{\ast }$and $z\in \lbrack -1,1]_{1,p+1,A}.$ Let $\widehat{z}$ be the
closest point of $[-1,1]$ to $z.$ We have the following inequalities :%
\begin{eqnarray*}
&&\varrho (g\circ P_{n}(z),[-1,1]) \\
&\leq &|g\circ P_{n}(z)-g\circ P_{n}(\widehat{z})| \\
&\leq &\overset{+\infty }{\underset{j=1}{\sum }}\frac{|g^{(j)}(P_{n}(%
\widehat{z}))|}{j!}|P_{n}(z)-P_{n}(\widehat{z})|^{j} \\
&\leq &\overset{+\infty }{\underset{j=1}{\sum }}\left( \overset{+\infty }{%
\underset{m=1}{\sum }}\frac{|P_{n}^{(m)}(\widehat{z})|}{m!}|z-\widehat{z}%
)|^{m}\right) ^{j} \\
&\leq &\overset{+\infty }{\underset{j=1}{\sum }}\left( \overset{+\infty }{%
\underset{m=1}{\sum }}\varrho (z,[-1,1])^{m}\right) ^{j} \\
&\leq &\frac{A}{p+1-2A} \\
&<&\frac{A}{p}
\end{eqnarray*}%
It follows that : 
\begin{equation*}
g\circ P_{n}([-1,1]_{1,p+1,A})\subset \lbrack -1,1]_{1,p,A},\text{ }p\in 
\mathbb{N}
^{\ast },\text{ }n\in 
\mathbb{N}
^{\ast }
\end{equation*}%
Thence the sequence $(g\circ P_{n})_{n\geq 1}$ has the $E(1)$ property.

2- For every $n\in 
\mathbb{N}
^{\ast },g_{n}$ is an entire function such that $g_{n}([-1,1])\subset
\lbrack -1,1].$ Let us show by induction that $\lambda (g_{n})\leq 1,$for
every $n\in 
\mathbb{N}
^{\ast }.$We have $\lambda (g_{1})=\lambda (g)\leq 1.$Assume that $\lambda
(g_{n})\leq 1$ for a certain $n\geq 1.$Then by virtue of Faa-di-Bruno
formula we have for every $x\in \lbrack -1,1]$ and $p\in 
\mathbb{N}
^{\ast }:$%
\begin{eqnarray*}
\frac{|g_{n+1}^{(p)}(x)|}{p!} &\leq &\frac{1}{2^{p}}\underset{%
j_{1}+2j_{2}+...+pj_{p}=p}{\sum }\frac{(j_{1}+j_{2}+...+j_{p})!}{%
j_{1}!j_{2}!...j_{p}!} \\
&&\frac{|g^{(j_{1}+j_{2}+...+j_{p})}(g_{n}(\frac{x}{2}))|}{%
(j_{1}+j_{2}+...+j_{p})!}\underset{s=1}{\overset{p}{\prod }}\left( \frac{%
|g_{n}^{(s)}(\frac{x}{2})|}{s!}\right) ^{j_{s}} \\
&\leq &\frac{1}{2^{p}}\underset{j_{1}+2j_{2}+...+pj_{p}=p}{\sum }\frac{%
(j_{1}+j_{2}+...+j_{p})!}{j_{1}!j_{2}!...j_{p}!} \\
&\leq &1
\end{eqnarray*}%
It follows that $\lambda (g_{n+1})\leq 1.$Thence according to the previous
part of this proposition the sequence $(g_{n})_{n\geq 1}$ verify the $E(1)$
property.

The proof of the proposition is then complete.
\end{proof}

\begin{example}
\end{example}

Direct computations show that $\lambda (\sin )=1.$ It follows that the
linear functional equation :

\begin{equation*}
\Phi (x)-\frac{1}{2}\Phi (\sin x)=-x
\end{equation*}%
has a unique solution which belongs to the Gevrey class $G_{1}([-1,1]).$\ 

\begin{example}
\end{example}

Let $(P_{n})_{n\geq 1}$be a sequence of holomorphic functions on $%
[-1,1]_{\sigma }(\sigma >0)$ such that we have for every $n\in 
\mathbb{N}
^{\ast }:$%
\begin{equation*}
\left \{ 
\begin{array}{c}
P_{n}([-1,1])\subset \lbrack -1,1] \\ 
\lambda (P_{n})\leq 1%
\end{array}%
\right.
\end{equation*}%
Let $g$ be an entire function such that :%
\begin{equation*}
\left \{ 
\begin{array}{c}
g([-1,1])\subset \lbrack -1,1] \\ 
\lambda (g)\leq 1%
\end{array}%
\right.
\end{equation*}%
Let $(a_{n})_{n\geq 1}$be a sequence of holomorphic functions on $%
[-1,1]_{\sigma }$ such that : 
\begin{equation*}
\overset{+\infty }{\underset{n=1}{\sum }}||a_{n}||_{\infty ,[-1,1]_{\sigma
}}<1
\end{equation*}%
Then it follows from corollary 4 and proposition 6 that the functional
equation : 
\begin{equation*}
\Phi (x)-\overset{+\infty }{\underset{n=1}{\sum }}a_{n}(x)\Phi
(g(P_{n}(x)))=u(x)
\end{equation*}%
has for every $u\in O([-1,1])$ a unique solution which belongs to the Gevrey
class $G_{1}([-1,1]).$

\begin{example}
\end{example}

Let $(\alpha _{n})_{n\geq 1}$be a sequence of real numbers. $f_{n}$ denotes
for every $n\in 
\mathbb{N}
^{\ast }$ the entire function defined by : 
\begin{equation*}
f_{n}(z):=\sin (z-\alpha _{n}),\text{ }z\in 
\mathbb{C}%
\end{equation*}%
Direct computations show that : 
\begin{equation*}
\lambda (f_{n})\leq 1,\text{ }n\in 
\mathbb{N}
^{\ast }
\end{equation*}%
Thence it follows from the previous example that the linear functional
equation :%
\begin{equation*}
\Phi (x)-\overset{+\infty }{\underset{n=1}{\sum }}\frac{x^{2}}{%
2^{n+1}(x^{2}+1)}\Phi (\sin (\sin (x-\alpha _{n})))=u(x)
\end{equation*}%
has for every $u\in O([-1,1])$ a unique solution which belongs to the Gevrey
class $G_{1}([-1,1]).$

\begin{example}
\end{example}

According to the proposition 6 above and to the fact that $\lambda (\sin
)=1, $ it follows that the sequence of functions $(g_{n})_{n\geq 1}$ defined
by : 
\begin{equation*}
g_{n}(z):=\sin ^{\left \langle n\right \rangle }\left( \frac{x}{2^{n-1}}%
\right) ,\text{ }z\in 
\mathbb{C}%
\end{equation*}%
verify the $E(1)$ property. Consequently the linear equation :

\begin{equation*}
\Phi (x)-\overset{+\infty }{\underset{n=1}{\sum }}\frac{\cos (\varepsilon
_{n}x)}{2^{n+1}}\text{ }\Phi \left( \sin ^{\left \langle n\right \rangle
}\left( \frac{x}{2^{n-1}}\right) \right) =u(x)
\end{equation*}%
has for every $u\in O([-1,1])$ and every bounded sequences $(\varepsilon
_{n})_{n\geq 1}$ of real numbers a unique solution which belongs to the
Gevrey class $G_{1}([-1,1]).$

\end{document}